\documentclass[12pt,a4paper]{article}
\usepackage{a4wide}

\usepackage[T1]{fontenc}
\usepackage{lmodern}
\usepackage{amsthm,amsmath,amssymb}
\usepackage{hyperref}
\usepackage{microtype}
\usepackage{hyperref}
\usepackage{graphicx}

\newtheorem{theorem}{Theorem}
\newtheorem{lemma}[theorem]{Lemma}

\newtheorem{proposition}[theorem]{Proposition}

\newtheorem{problem}[theorem]{Problem}

\def\inst#1{$^{#1}$}

\begin{document}

\title{Edge-ordered Ramsey numbers\footnote{An extended abstract of this paper appeared in the proceeding of Eurocomb 2019~\cite{bv19}.}}

\author{Martin Balko\inst{1} \thanks{The first author was supported by the grant no.~18-13685Y of the Czech Science Foundation (GA\v{C}R) and by the Center for Foundations of Modern Computer Science (Charles University project UNCE/SCI/004). This article is part of a project that has received funding from the European Research Council (ERC) under the European Union's Horizon 2020 research and innovation programme (grant agreement No 810115).} 
\and
M\'{a}t\'{e} Vizer\inst{2} \thanks{The second author was supported by the Hungarian National Research, Development and Innovation Office -- NKFIH under the grant SNN 129364 and KH 130371, by the J\'anos Bolyai Research Fellowship of the Hungarian Academy of Sciences and by the New National Excellence Program under the grant number \'UNKP-19-4-BME-287.} 
}

\maketitle

\begin{center}
{\footnotesize
\inst{1} 
Department of Applied Mathematics, \\
Faculty of Mathematics and Physics, Charles University, Prague, Czech Republic \\
\texttt{balko@kam.mff.cuni.cz}
\\\ \\
\inst{2} 
Alfr\'{e}d R\'{e}nyi Institute of Mathematics,\\ Hungarian Academy of Sciences, Budapest, Hungary \\
\texttt{vizermate@gmail.com}
}
\end{center}

\begin{abstract}
We introduce and study a variant of Ramsey numbers for \emph{edge-ordered graphs}, that is, graphs with linearly ordered sets of edges.
The \emph{edge-ordered Ramsey number} $\overline{R}_e(\mathfrak{G})$ of an edge-ordered graph $\mathfrak{G}$ is the minimum positive integer $N$ such that there exists an edge-ordered complete graph $\mathfrak{K}_N$ on $N$ vertices such that every 2-coloring of the edges of $\mathfrak{K}_N$ contains a monochromatic copy of $\mathfrak{G}$ as an edge-ordered subgraph of $\mathfrak{K}_N$.

We prove that the edge-ordered Ramsey number $\overline{R}_e(\mathfrak{G})$ is finite for every edge-ordered graph $\mathfrak{G}$ and we obtain better estimates for special classes of edge-ordered graphs.
In particular, we prove $\overline{R}_e(\mathfrak{G}) \leq 2^{O(n^3\log{n})}$ for every bipartite edge-ordered graph $\mathfrak{G}$ on $n$ vertices.
We also introduce a natural class of edge-orderings, called \emph{lexicographic edge-orderings}, for which we can prove much better upper bounds on the corresponding edge-ordered Ramsey numbers.
\end{abstract}

\section{Introduction}

An \emph{edge-ordered graph} $\mathfrak{G} = (G,\prec)$ consists of a graph $G=(V,E)$ and a linear ordering $\prec$ of the set of edges~$E$.
We sometimes use the term \emph{edge-ordering of $G$} for the ordering $\prec$ and also for $\mathfrak{G}$.
An edge-ordered graph $(G,\prec_1)$ is an \emph{edge-ordered subgraph} of an edge-ordered graph $(H,\prec_2)$ if $G$ is a subgraph of $H$ and $\prec_1$ is a suborder of $\prec_2$.
We say that $(G,\prec_1)$ and $(H,\prec_2)$ are \emph{isomorphic} if there is a graph isomorphism between $G$ and $H$ that also preserves the edge-orderings $\prec_1$ and $\prec_2$.

For a positive integer $k$, a \emph{$k$-coloring} of the edges of a graph $G$ is any function that assigns one of the $k$ \emph{colors} to each edge of $G$.
The \emph{edge-ordered Ramsey number} $\overline{R}_e(\mathfrak{G})$ of an edge-ordered graph $\mathfrak{G}$ is the minimum positive integer $N$ such that there exists an edge-ordering $\mathfrak{K}_N$ of the complete graph $K_N$ on $N$ vertices such that every 2-coloring of the edges of $\mathfrak{K}_N$ contains a monochromatic copy of $\mathfrak{G}$ as an edge-ordered subgraph of $\mathfrak{K}_N$.

More generally, for two edge-ordered graphs $\mathfrak{G}$ and $\mathfrak{H}$, we use $\overline{R}_e(\mathfrak{G},\mathfrak{H})$ to denote the minimum positive integer $N$ such that there exists an edge-ordering $\mathfrak{K}_N$ of $K_N$ such that every 2-coloring of the edges of $\mathfrak{K}_N$ with colors red and blue contains a red copy of $\mathfrak{G}$ or a blue copy of $\mathfrak{H}$ as an edge-ordered subgraph of $\mathfrak{K}_N$.
We call the number $\overline{R}_e(\mathfrak{G},\mathfrak{H})$ the \emph{non-diagonal} edge-ordered Ramsey number.

To our knowledge, Ramsey numbers of edge-ordered graphs were not considered in the literature.
On the other hand, Ramsey numbers of graphs with ordered vertex sets have been quite extensively studied recently; for example, see~\cite{bckk15,bjv16,clfs17}. 
For questions concerning extremal problems about vertex-ordered graphs consult the recent surveys~\cite{t18,t19}.
A \emph{vertex-ordered graph} $\mathcal{G}=(G,\prec)$ (or simply an \emph{ordered graph}) is a graph $G$ with a fixed linear ordering $\prec$ of its vertices.
We use the term \emph{vertex-ordering of $G$} to denote the ordering $\prec$ as well as the ordered graph~$\mathcal{G}$.
An ordered graph $(G,\prec_1)$ is a \emph{vertex-ordered subgraph} of an ordered graph $(H,\prec_2)$ if $G$ is a subgraph of $H$ and $\prec_1$ is a suborder of $\prec_2$.
We say that $(G,\prec_1)$ and $(H,\prec_2)$ are \emph{isomorphic} if there is a graph isomorphism between $G$ and $H$ that also preserves the vertex-orderings $\prec_1$ and $\prec_2$.
Unlike in the case of edge-ordered graphs, there is a unique vertex-ordering $\mathcal{K}_N$ of $K_N$ up to isomorphism.
The \emph{ordered Ramsey number} $\overline{R}(\mathcal{G})$ of an ordered graph $\mathcal{G}$ is the minimum positive integer $N$ such that every 2-coloring of the edges of $\mathcal{K}_N$ contains a monochromatic copy of $\mathcal{G}$ as a vertex-ordered subgraph of $\mathcal{K}_N$.

For an $n$-vertex graph $G$, let $R(G)$ be the Ramsey number of $G$.
It is easy to see that $R(G) \leq \overline{R}(\mathcal{G})$ and $R(G) \leq  \overline{R}_e(\mathfrak{G})$ for each vertex-ordering $\mathcal{G}$ of $G$ and edge-ordering $\mathfrak{G}$ of~$G$.
We also have $\overline{R}(G) \leq  \overline{R}(\mathcal{K}_n) = R(K_n)$ and thus ordered Ramsey numbers are always finite.
Proving that $\overline{R}_e(\mathfrak{G})$ is always finite seems to be more challenging; see Theorem~\ref{thm-edgeOrderedRamseyFinite}.

The Tur\'an numbers of edge-ordered graphs were recently introduced in~\cite{gmnptv19}, motivated by a lemma in~\cite[Lemma 23]{gpv18}.
The authors of~\cite{gmnptv19} proved, for example, a variant of the Erd\H{o}s--Stone--Simonovits Theorem for edge-ordered graphs, and also investigated the Tur\'an numbers of small edge-ordered paths, star forests, and 4-cycles; see also the last section of~\cite{t19}.

Another related problem is to determine the maximum length of a monotone increasing path that must appear in any edge-ordered complete graph on $n$ vertices.
The Chv\'atal--Koml\'os conjecture~\cite{ck1971} says that this quantity is linear in $n$ and the authors of~\cite{bkpstw2018} could prove an almost linear lower bound.

For $n \in \mathbb{N}$, we use $[n]$ to denote the set $\{1,\dots,n\}$.
We omit floor and ceiling signs whenever they are not crucial.
All logarithms in this paper are base $2$.

\section{Our results}
\label{sec-ourResults}

We study the growth rate of edge-ordered Ramsey numbers with respect to the number of vertices for various classes of edge-ordered graphs.
As our first result, we show that edge-ordered Ramsey numbers are always finite and thus well-defined.

\begin{theorem}
\label{thm-edgeOrderedRamseyFinite}
For every edge-ordered graph $\mathfrak{G}$, the edge-ordered Ramsey number $\overline{R}_e(\mathfrak{G})$ is finite.
\end{theorem}

Theorem~\ref{thm-edgeOrderedRamseyFinite} also follows from a recent deep result of Hubi\v{c}ka and Ne\v{s}et\v{r}il~\cite[Theorem~4.33]{hubNes16} about Ramsey numbers of general relational structures.
In comparison, our proof of Theorem~\ref{thm-edgeOrderedRamseyFinite} is less general, but it is much simpler and produces better and more explicit bound on $\overline{R}_e(\mathfrak{G})$.
It is a modification of the proof of Theorem~12.13~\cite[Page~138]{promel13}, which is based on the Graham--Rothschild Theorem~\cite{grahamRotschild71}.
In fact, the proof of Theorem~\ref{thm-edgeOrderedRamseyFinite} yields a stronger induced-type statement where additionally the ordering of the vertex set is fixed; see Theorem~\ref{thm-finiteInduced}.
Theorem~\ref{thm-edgeOrderedRamseyFinite} can also be extended to $k$-colorings with $k >2$.

Due to the use of the Graham--Rothschild Theorem, the bound on the edge-ordered Ramsey numbers obtained in the proof of Theorem~\ref{thm-edgeOrderedRamseyFinite} is still enormous.
It follows from a result of Shelah~\cite[Theorem~2.2]{shelah88} that this bound on $\overline{R}_e(\mathfrak{G})$ is primitive recursive, but it grows faster than, for example, a tower function of any fixed height.
Thus we aim to prove more reasonable estimates on edge-ordered Ramsey numbers, at least for some classes of edge-ordered graphs.

As our second main result, we show that one can obtain a much better upper bound on non-diagonal edge-ordered Ramsey numbers of two edge-ordered graphs, provided that one of them is bipartite.
For $d \in \mathbb{N}$, we say that a graph $G$ is \emph{$d$-degenerate} if every subgraph of $G$ has a vertex of degree at most $d$. 

\begin{theorem}
\label{thm-superexponentialBound}
Let $\mathfrak{H}$ be a $d$-degenerate edge-ordered graph on $n'$ vertices and let $\mathfrak{G}$ be a bipartite edge-ordered graph with $m$ edges and with both parts containing $n$ vertices.
If $d \leq n$ and $n' \leq t^{d+1}$ for $t=3n^{10}m!$, then
\[\overline{R}_e(\mathfrak{H},\mathfrak{G}) \leq (n')^2t^{d+1}.\]
\end{theorem}

In particular, if $\mathfrak{G}$ is a bipartite edge-ordered graph on $n$ vertices, then $\overline{R}_e(\mathfrak{G}) \leq 2^{O(n^3\log{n})}$.
We believe that the bound can be improved.
In fact, it is possible that $\overline{R}_e(\mathfrak{G})$ is at most exponential in the number of vertices of $\mathfrak{G}$ for every edge-ordered graph $\mathfrak{G}$; see Section~\ref{sec-openProblems} for more open problems.
We note that, for every graph $G$ and its vertex-ordering $\mathcal{G}$, both the standard Ramsey number $R(G)$ and the ordered Ramsey number $\overline{R}(\mathcal{G})$ grow at most exponentially in the number of vertices of $G$. 

In general, the difference between edge-ordered Ramsey numbers and ordered Ramsey numbers with the same underlying graph can be very large.
Let $M_n$ be a \emph{matching} on $n$ vertices, that is, a graph formed by a collection of $n/2$ disjoint edges.
There are ordered matchings $\mathcal{M}_n=(M_n,<)$  with super-polynomial ordered Ramsey numbers $\overline{R}(\mathcal{M}_n)$ in $n$~\cite{bckk15,clfs17}.
In fact this is true for almost all ordered matchings on $n$ vertices~\cite{clfs17}.
On the other hand, all edge-orderings of $M_n$ are isomorphic as edge-ordered graphs and thus $\overline{R}_e(\mathfrak{M}_n) = R(M_n) \leq O(n)$ for every edge-ordering $\mathfrak{M}_n$ of $M_n$.

In Section~\ref{sec-lexicographic}, we consider a special class of edge-orderings, which we call \emph{lexicographic edge-orderings}, for which we can prove much better upper bounds on their edge-ordered Ramsey numbers and which seem to be quite natural.

An ordering $\prec$ of edges of a graph $G=(V,E)$ is \emph{lexicographic} if there is a one-to-one correspondence $f \colon V \to \{1,\dots,|V|\}$ such that any two edges $\{u,v\}$ and $\{w,t\}$ of~$G$ with $f(u)<f(v)$ and $f(w)<f(t)$ satisfy $\{u,v\} \prec \{w,t\}$ if either $f(u) < f(w)$ or if $(f(u)=f(w) \;\&\; f(v) < f(t))$.
We say that such mapping $f$ is \emph{consistent} with $\prec$.
Note that, for every vertex $u$, the edges $\{u,v\}$ with $f(u) < f(v)$ form an interval in $\prec$.
Also observe that there is a unique (up to isomorphism) lexicographic edge-ordering $\mathfrak{K}^{lex}_n$ of $K_n$. Setting $\{u,v\} \prec' \{w,t\}$ if either $f(u) < f(w)$ or if $(f(u)=f(w) \;\&\; f(v) > f(t))$ we obtain the \emph{max-lexicographic} edge-ordering $\prec'$ of $G$.
Observe that in the max-lexicographic ordering, for every vertex $u$, the edges $\{u,v\}$ with $f(u) < f(v)$ again form an interval in $\prec'$.
When compared to the lexicographic edge-ordering, each of these intervals is reversed, but the ordering of the intervals is kept the same.

For a linear ordering $<$ on some set $X$, we use $<^{-1}$ to denote the \emph{inverse ordering} of $<$, that is, for all $x,y \in X$, we have $x <^{-1} y$ if and only if $y < x$.

The lexicographic and max-lexicographic edge-orderings are natural, as Ne\v{s}et\v{r}il and R\"{o}dl~\cite{nesetrilRodl17} showed that these orderings are canonical in the following sense. 

\begin{theorem}[\cite{nesetrilRodl17}]
\label{thm-canonical}
For every $n \in \mathbb{N}$, there is a positive integer $T(n)$ such that every edge-ordered complete graph on $T(n)$ vertices contains a copy of $K_n$ such that the edges of this copy induce one of the following four edge-orderings: lexicographic edge-ordering $\prec$, max-lexicographic edge-ordering $\prec'$, $\prec^{-1}$, or $(\prec')^{-1}$. 
\end{theorem}

Theorem~\ref{thm-canonical} is also an unpublished result of Leeb; see~\cite{nprv85}.
It is thus natural to consider the following variant of edge-ordered Ramsey numbers, which turns out to be more tractable than general edge-ordered Ramsey numbers.
The \emph{lexicographic edge-ordered Ramsey number} $\overline{R}_{lex}(\mathfrak{G})$ of a lexicographically edge-ordered graph $\mathfrak{G}$ is the minimum $N$ such that every 2-coloring of the edges of the lexicographically edge-ordered complete graph $\mathfrak{K}_N^{lex}$ on $N$ vertices contains a monochromatic copy of~$\mathfrak{G}$ as an edge-ordered subgraph of $\mathfrak{K}_N^{lex}$.
Observe that $\overline{R}_e(\mathfrak{G}) \leq \overline{R}_{lex}(\mathfrak{G})$ for every lexicographically edge-ordered graph $\mathfrak{G}$.

For every lexicographically edge-ordered graph $\mathfrak{G}=(G,\prec)$, the lexicographic edge-ordered Ramsey number $\overline{R}_{lex}(\mathfrak{G})$ can be estimated from above with the ordered Ramsey number of some vertex-ordering of $G$.
More specifically, we have the following result.

\begin{lemma}
\label{lem-lexicographicVertex}
Every lexicographically edge-ordered graph $\mathfrak{G}=(G,\prec)$ satisfies
\[\overline{R}_{lex}(\mathfrak{G}) \leq \min_f\overline{R}(\mathcal{G}_f),\]
where the minimum is taken over all one-to-one correspondences $f \colon V \to \{1,\dots,|V|\}$ that are consistent with the lexicographic edge-ordering $\mathfrak{G}$ and $\mathcal{G}_f$ is the vertex-ordering of $G$ determined by $f$.
\end{lemma}

We prove Lemma~\ref{lem-lexicographicVertex} in Section~\ref{sec-lexicographic}.
Since $\overline{R}(\mathcal{K}_n) = R(K_n)$, it follows from Lemma~\ref{lem-lexicographicVertex} and from the well-known bound $R(K_n) \leq 2^{2n}$ by Erd\H{o}s and Szekeres~\cite{erdosSzekeres35} that the numbers $\overline{R}_{lex}(\mathfrak{G})$ are always at most exponential in the number of vertices of $G$.
In fact, we have $\overline{R}_{lex}(\mathfrak{K}^{lex}_n) = \overline{R}(\mathcal{K}_n)  = R(K_n)$ for every $n$.
The equality is achieved in the statement of Lemma~\ref{lem-lexicographicVertex}, for example, for graphs with a unique vertex-ordering determined by the lexicographic edge-ordering.
Such graphs include graphs where each edge is contained in a triangle. 
Additionally, combining Lemma~\ref{lem-lexicographicVertex} with a result of Conlon et al.~\cite[Theorem~3.6]{clfs17} gives the estimate
\[\overline{R}_{lex}(\mathfrak{G}) \leq 2^{O(d\log^2{(2n/d)})}\]
for every $d$-degenerate lexicographically edge-ordered graph $\mathfrak{G}$ on $n$ vertices.
In particular, $\overline{R}_{lex}(\mathfrak{G})$ is at most quasi-polynomial in $n$ if $d$ is fixed.

We note that the bound in Lemma~\ref{lem-lexicographicVertex} is not always tight.
For example, $R(K_{1,n}) = \overline{R}_{lex}(\mathfrak{K}_{1,n})$ for every edge-ordering $\mathfrak{K}_{1,n}$ of $K_{1,n}$, as any two edge-ordered stars $K_{1,n}$ are isomorphic as edge-ordered graphs.
However, the Ramsey number $R(K_{1,n})$ is known to be strictly smaller than $\overline{R}(\mathcal{K}_{1,n})$ for $n$ even and for any vertex-ordering $\mathcal{K}_{1,n}$ of $K_{1,n}$; see~\cite{burrRoberts73} and~\cite[Observation~11 and Theorem~12]{bckk13}.

As an application of Lemma~\ref{lem-lexicographicVertex} we obtain asymptotically tight estimate on the following lexicographic edge-ordered Ramsey numbers of paths.
The \emph{edge-monotone path} $\mathfrak{P}_n=(P_n,\prec)$ is the edge-ordered path on $n$ vertices $v_1,\dots,v_n$, where $\{v_1,v_2\} \prec \dots \prec \{v_{n-1},v_n\}$; see part~(a) of Figure~\ref{fig-monotone}.

\begin{figure}
\begin{center}
\includegraphics{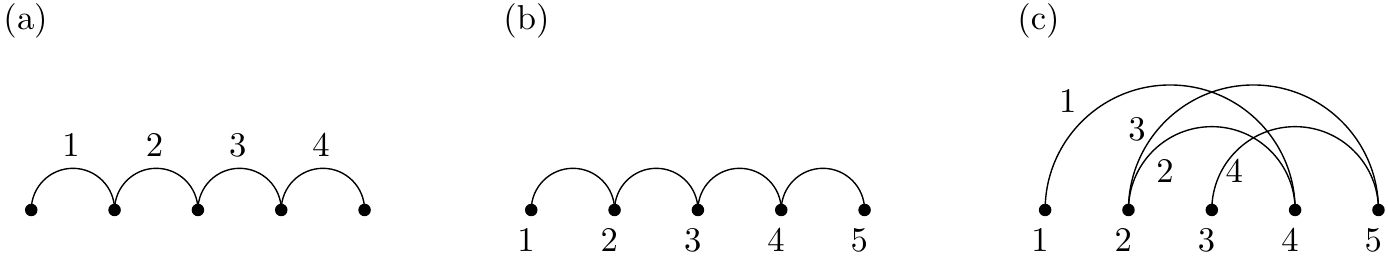}
\caption{(a) The edge-monotone path on $5$ vertices. 
(b) The monotone path on $5$ vertices. 
(c) A different ordered path on $5$ vertices and the corresponding lexicographic edge-ordering.
The label of each edge and vertex denotes the position in the edge- and vertex-ordering, respectively.}
\label{fig-monotone}
\end{center}
\end{figure}

\begin{proposition}
\label{prop-lexicographicPath}
For every integer $n >2$, we have \[\overline{R}_{lex}(\mathfrak{P}_n) \leq 2n-3+\sqrt{2n^2-8n+11}.\]
\end{proposition}

The proof of Proposition~\ref{prop-lexicographicPath} uses the fact that the one-to-one correspondence $f$ consistent with the lexicographic edge-ordering of $P_n$ is not determined uniquely.
Indeed, we can choose the mapping $f$ so that it determines the vertex-ordering $\mathcal{P}_n$ of $P_n$ where edges are between consecutive pairs of vertices.
Such vertex-ordering $\mathcal{P}_n$ is called \emph{monotone path}; see part~(b) of Figure~\ref{fig-monotone}. 
However, it is known that $\overline{R}(\mathcal{P}_n) = (n-1)^2+1$~\cite{choPonn02} and thus we cannot apply Lemma~\ref{lem-lexicographicVertex} to this ordering to obtain a linear bound on $\overline{R}_{lex}(\mathfrak{P}_n)$. 
Instead we choose a different mapping $f$ that determines a vertex-ordering of $P_n$ with linear ordered Ramsey number; see part~(c) of Figure~\ref{fig-monotone}.

As our last result, we show an upper bound on edge-ordered Ramsey numbers of two graphs, where one of them is bipartite and suitably lexicographically edge-ordered.
This result uses a stronger assumption about $\mathfrak{G}$ than Theorem~\ref{thm-superexponentialBound}, but gives much better estimate.
For $m,n \in \mathbb{N}$, let $\mathfrak{K}^{lex}_{m,n}$ be the lexicographic edge-ordering of $K_{m,n}$ that induces a vertex-ordering, in which both parts of $K_{m,n}$ form an interval. 

\begin{theorem}
\label{thm-polynomialBound}
Let $\mathfrak{H}$ be a $d$-degenerate edge-ordered graph on $n'$ vertices and let $\mathfrak{G}$ be an edge-ordered subgraph of $\mathfrak{K}^{lex}_{n,n}$. 
Then
\[\overline{R}_e(\mathfrak{H},\mathfrak{G}) \leq (n')^2n^{d+1}.\]
\end{theorem}

The proof of Theorem~\ref{thm-edgeOrderedRamseyFinite} is presented in Section~\ref{sec-edgeOrderedRamseyFinite}.
Both Theorem~\ref{thm-superexponentialBound} and Theorem~\ref{thm-polynomialBound} are proved in Section~\ref{sec-betterBounds}.
Section~\ref{sec-lexicographic} contains the proofs of Lemma~\ref{lem-lexicographicVertex} and Proposition~\ref{prop-lexicographicPath}.
Finally, we mention some open problems in Section~\ref{sec-openProblems}.

\section{Proof of Theorem~\ref{thm-edgeOrderedRamseyFinite}}
\label{sec-edgeOrderedRamseyFinite}

In this section, we prove Theorem~\ref{thm-edgeOrderedRamseyFinite} by showing that edge-ordered Ramsey numbers are always finite.
The proof is carried out using the Graham--Rothschild Theorem~\cite{grahamRotschild71}.
To state this result, we need to introduce some definitions first.
We follow the notation from~\cite{promel13}.

Let $N$ and $t$ be nonnegative integers with $t \leq N$ and let $A$ be a finite set of symbols not containing the symbols $\lambda_1,\dots,\lambda_t$.
Then the set $[A]\binom{N}{t}$ of \emph{$t$-parameter words of length $N$ over $A$} is the set of mappings $f \colon [N] \to A \cup \{\lambda_1,\dots,\lambda_t\}$ such that for every $j$ with $1 \leq j \leq t$ there exists $i \in [N]$ such that $f(i)=\lambda_j$ and $\min(f^{-1}(\lambda_i)) < \min(f^{-1}(\lambda_j))$ for all $i$ and $j$ with $1 \leq i < j \leq t$.
The \emph{composition} $f \cdot g \in [A]\binom{N}{r}$ of $f \in [A]\binom{N}{t}$ and $g \in [A]\binom{t}{r}$ is defined by
\[(f \cdot g)(i) = \left.
\begin{cases}
f(i), & \text{ if } f(i) \in A \text{ and} \\
g(j), & \text{ if } f(i) = \lambda_j.
\end{cases}
\right.
\]

The following result, called the Graham--Rothschild Theorem, is a strengthening of the famous Hales--Jewett Theorem.

\begin{theorem}[The Graham--Rothschild Theorem~\cite{grahamRotschild71,promel13}]
\label{thm-grahamRotschild}
Let $A$ be a finite alphabet and let $r,t \in \mathbb{N}_0$ and $k \in \mathbb{N}$ be integers such that $r \leq t$.
Then there is a positive integer $N=GR(|A|,k,r,t)$ such that for every $k$-coloring $\chi'$ of $[A]\binom{N}{r}$ there exists a \emph{monochromatic} $f \in [A]\binom{N}{t}$, that is, $f$ satisfies
\[\chi'(f \cdot g) = \chi'(f \cdot h)\]
for all $g,h \in [A]\binom{t}{r}$.
\end{theorem}

Applying the Graham--Rothschild Theorem similarly as in~\cite[Page~138]{promel13}, we can derive the following result, which is actually a stronger statement than Theorem~\ref{thm-edgeOrderedRamseyFinite}.

\begin{theorem}
\label{thm-finiteInduced}
Let $(F,\prec_v,\prec_e)$ be a graph with linear orderings $\prec_v$ and $\prec_e$ on its vertices and edges, respectively.
Then, for every $k \in \mathbb{N}$, there exists a graph $G$ with orderings $\lessdot_v$ and $\lessdot_e$ of its vertices and edges, respectively, such that for every $k$-coloring of the edges of~$G$ there is a monochromatic induced copy of $(F,\prec_v,\prec_e)$ in $(G,\lessdot_v,\lessdot_e)$.
\end{theorem}

Note that we can fix the vertex-ordering of the monochromatic copy as well as the edge-ordering.
Moreover, the obtained monochromatic copy of $F$ is contained in the large graph $G$ as an induced subgraph.
Unfortunately, as we discussed in Section~\ref{sec-ourResults}, the obtained bound on the number of vertices of $G$ is enormous.

\begin{proof}[Proof of Theorem~\ref{thm-finiteInduced}]
We use $n$ and $m$ to denote the number of vertices and edges of $F$, respectively.
Let $G$ be defined as follows.
Choose $N \in \mathbb{N}$ such that $N \geq GR(1,k,3,n+m)$.
Let $2^{[N]}$ be the vertex set of $G$ and let $\{X,Y\}$ with $X,Y \subseteq [N]$ be an edge of $G$ if $X \cap Y \neq \emptyset$.
For any two sets $X,Y \subseteq [N]$ with $\min(X) < \min(Y)$, we set $X \lessdot_v' Y$.
Note that $\lessdot'_v$ is a partial ordering on the vertices of $G$.
We let $\lessdot_v$ be an arbitrary linear extension of $\lessdot'_v$.
For two edges $\{X,Y\}$ and $\{U,V\}$ of $G$ with $\min(X \cap Y) < \min(U \cap V)$, we set $\{X,Y\} \lessdot_e' \{U,V\}$.
Observe that $\lessdot_e'$ defines a partial ordering of the edges of $G$.
We let $\lessdot_e$ be an arbitrary linear extension of $\lessdot'_e$.

We show that $(G,\lessdot_v,\lessdot_e)$ satisfies the statement of the theorem.
We use $v_1 \prec_v \cdots \prec_v v_n$ and $f_1 \prec_e \cdots \prec_e f_m$ to denote the vertices and edges of $F$, respectively.
Let $\chi$ be a $k$-coloring of the edges of $G$.
We use $\chi$ to define a $k$-coloring $\chi'$ of $3$-parameter words of length $N$ over the single-letter alphabet $\{0\}$.
Given such  a word $w \in [\{0\}]\binom{N}{3}$, we set $S_i$, $i \in \{1,2,3\}$, to be the set of positions from $[N]$ on which $w$ contains the $i$th variable symbol $\lambda_i$.
Note that the sets $S_1,S_2,S_3$ are pairwise disjoint, non-empty and, since the first occurrence of $\lambda_i$ precedes the first occurrence of $\lambda_j$ in $w$ for all $i<j$, we also have $\min(S_1)<\min(S_2)<\min(S_3)$.
Setting $X = S_1 \cup S_3$ and $Y = S_2 \cup S_3$, we have two vertices of $G$ that form an edge of $G$ and that satisfy $X \lessdot_v Y$.
We let $\chi'(w) = \chi(\{X,Y\})$.

By the Graham--Rothschild Theorem (Theorem~\ref{thm-grahamRotschild}) applied with $t=n+m$ and $r=3$, there is an $(n+m)$-parameter word $w \in [\{0\}]\binom{N}{n+m}$ such that $\chi'(w \cdot v) = \chi'(w \cdot v')$ for all $v,v' \in [\{0\}]\binom{n+m}{3}$.
Let $b$ be the common color of the words $w \cdot v$ in $\chi'$.
Similarly as before, for every $i \in [n+m]$, we let $S_i \subseteq [N]$ be the set of positions on which $w$ contains the $i$th variable symbol $\lambda_i$.
Again, observe that the sets $S_i$ are pairwise disjoint and satisfy $\min(S_1) < \cdots < \min(S_{n+m})$.
For every $i \in [n]$, we let 
\[F_i = S_i \cup \bigcup_{\substack{j \colon j \in [m],\\ v_i \in f_j}} S_{n+j}.\]

The sets $F_1,\dots,F_n$ then induce a vertex-ordered and edge-ordered graph $(F^*,\prec_v,\prec_e)$ in~$G$.
We show that $F^*$ is a copy of $(F,\prec_v,\prec_e)$ in $(G,\lessdot_v,\lessdot_e)$.
Indeed, since $\min(S_1) < \cdots < \min(S_{n+m})$, we have $\min(F_i) = \min(S_i) < \min(S_j) = \min(F_j)$ if $i<j$ and thus $F_i \lessdot_v F_j$ for all $v_i \prec_v v_j$.
Moreover, since the sets $S_1,\dots,S_{n+m}$ are pairwise disjoint, $\{F_i,F_j\}$ is an edge of $F^*$ if and only if there is an edge $f_l$ of $F$ with $f_l = \{v_i,v_j\}$, which gives $F$ as 
an induced subgraph of $G$.
This is because we have $F_i \cap F_j = S_{n+l}$.
Let $\{F_i,F_j\}$ and $\{F_{i'},F_{j'}\}$ be two edges of $F^*$.
Since $\{F_i,F_j\}$ and $\{F_{i'},F_{j'}\}$ are edges of $G$, the sets $S_{n+l} = F_i \cap F_j$ and $S_{n+l'} = F_{i'} \cap F_{j'}$ correspond to the edges $f_l=\{v_i,v_j\}$ and $f_{l'}=\{v_{i'},v_{j'}\}$ of $F$, respectively, by the definition of $F_1,\dots,F_n$.
Assume $f_l \prec_e f_{l'}$.
Then $\lambda_{n+l}$ precedes $\lambda_{n+l'}$, as $l<l'$, and thus $\min(S_{n+l})<\min(S_{n+l'})$.
It follows from the fact $S_{n+l} = F_i \cap F_j$ and $S_{n+l'} = F_{i'} \cap F_{j'}$ and from the definition of $\lessdot_e$ that $\{F_i,F_j\} \lessdot_e \{F_{i'},F_{j'}\}$ if and only if $f_l \prec_e f_{l'}$.

It remains to show that all edges of $F^*$ are monochromatic in $\chi$.
Let $\{F_i,F_j\}$ be an edge of $F^*$ with $i<j$ and let $S_{n+l} = F_i \cap F_j$.
Note that the sets $F_i \setminus S_{n+l}$, $F_j \setminus S_{n+l}$, and $S_{n+l}$ are nonempty, pairwise disjoint, and satisfy $\min(F_i \setminus S_{n+l}) < \min(F_j \setminus S_{n+l}) < \min(S_{n+l})$.
We let $v \in \{0,\lambda_1,\lambda_2,\lambda_3\}^{n+m}$ be the word with symbols $\lambda_1,\lambda_2,\lambda_3$ on positions from sets
$\{i\} \cup \{n+s \colon v_i \in f_s \neq f_l\}$, $\{j\} \cup \{n+s \colon v_j \in f_s \neq f_l\}$, and $\{n+l\}$ in $w$, respectively.
Then $v \in [\{0\}]\binom{n+m}{3}$ and $w \cdot v \in [\{0\}]\binom{N}{3}$ is the $3$-parameter word with variable symbols $\lambda_1,\lambda_2,\lambda_3$ on positions from $F_i \setminus S_{n+l}$, $F_j \setminus S_{n+l}$, and $S_{n+l}$, respectively.
By the choice of $w$, we have $b= \chi'(w \cdot v) = \chi(\{F_i,F_j\})$.
Thus all edges of $F^*$ have the color $b$ in~$\chi$.
\end{proof}

Now, we obtain Theorem~\ref{thm-edgeOrderedRamseyFinite} as a corollary of Theorem~\ref{thm-finiteInduced}.

\begin{proof}[Proof of Theorem~\ref{thm-edgeOrderedRamseyFinite}]
For a given edge-ordered graph $\mathfrak{G}=(G,\prec)$, let $<$ be an arbitrary ordering of the vertices of $G$.
By Theorem~\ref{thm-finiteInduced}, there is a graph $H$ and orderings $<'$ and $\prec'$ of its vertices and edges, respectively, such that every $2$-coloring of the edges of $H$ contains a monochromatic induced copy of $(G,<,\prec)$.
It thus suffices to consider edge-ordered $\mathfrak{K}_N$ that contains $(H,\prec')$ as an edge-ordered subgraph.
Then every 2-coloring of the edges of $\mathfrak{K}_N$ contains a monochromatic copy of $\mathfrak{G}$ as an edge-ordered subgraph (not necessarily induced).
\end{proof}

\section{Proofs of Theorems~\ref{thm-superexponentialBound} and~\ref{thm-polynomialBound}}
\label{sec-betterBounds}

In this section, we prove both Theorems~\ref{thm-superexponentialBound} and~\ref{thm-polynomialBound}.
That is, we derive a super-exponential upper bound on the edge-ordered Ramsey numbers $\overline{R}_e(\mathfrak{H},\mathfrak{G})$ of two edge-ordered graphs $\mathfrak{H}$ and $\mathfrak{G}$, where $\mathfrak{G}$ is bipartite. 
We then improve this bound under the additional assumption that $\mathfrak{G} \subseteq \mathfrak{K}^{lex}_{n,n}$.

As a first step, we prove the following lemma, which is used in proofs of both Theorem~\ref{thm-superexponentialBound} and~\ref{thm-polynomialBound}.
The proof of this lemma is inspired by a similar ``greedy-embedding'' approach used, for example, in~\cite{bckk13}.

\begin{lemma}
\label{lem-greedyEmbedding}
Let $H$ be a $d$-degenerate graph on $n'$ vertices and let $v_1 \lessdot \dots \lessdot v_{n'}$ be a vertex-ordering of $H$ such that each $v_j$ has at most $d$ neighbors $v_i$ with $i<j$.
Then, for every $t \in \mathbb{N}$, there is $K_N$ with $N=(n')^2t^{d+1}$ and with the vertex set partitioned into $n'$ sets $I_1,\dots,I_{n'}$ of the same size such that the following statement holds.
In every red-blue coloring of the edges of $K_N$, there is a blue copy of $H$ in $K_N$ with a copy of each $v_i$ in $I_i$ or a red copy of $K_{t,t}$ in $K_N$ with each part contained in a different set $I_i$.
\end{lemma}
\begin{proof}
Let $v_1 \lessdot \dots \lessdot v_{n'}$ be a vertex-ordering of $H$ such that each $v_j$ has at most $d$ neighbors $v_i$ with $i<j$.
Such an ordering exists, as $H$ is $d$-degenerate.
For $N = (n')^2t^{d+1}$ and for every vertex $v_i$ of $H$, let $I_i$ be a set of vertices such that $|I_i| = M = n't^{d+1}$ and let the disjoint union $I_1 \cup \cdots \cup I_{n'}$ be the vertex set of $K_N$.

Let $\chi$ be a red-blue coloring of the edges of $K_N$.
We now try to greedily embed a blue copy of $H$ on vertices $h(v_1),\dots,h(v_{n'})$ in $\chi$ such that $h(v_i) \in I_i$ for every $i \in [n']$.
We proceed so that if the embedding fails at some step, we obtain a red copy of $K_{t,t}$ in $\chi$ with each part contained in a different set $I_i$.

For each $i \in [n']$, let $C_i$ be a set of candidates for the vertex $h(v_i)$.
Initially, we set $C_i = I_i$.
We then proceed in steps $i=1,\dots,n'$, assuming that we have already determined the vertices $h(v_1),\dots,h(v_{i-1})$ in steps $1,\dots,i-1$, respectively.

In step $i$, assume that, for every neighbor $v_j$ of $v_i$ in $H$ with $i<j$, all but at most $t-1$ vertices from $C_i$ have at least $|C_j|/t$ blue neighbors in $C_j$.
In such a case, if $|C_i| \geq n't$, then there is a vertex of $C_i$ that has at least $|C_j|/t$ blue neighbors in each $C_j$ such that $v_j$ is a neighbor of $v_i$ in $H$ with $i<j$.
This is because $v_i$ has at most $n'-1$ neighbors $v_j$ in $H$ with $i<j$ and 
\[|C_i| - (n'-1)(t-1) \geq n't-(n'-1)(t-1) > 0.\]

We let $h(v_i)$ be an arbitrary such vertex from $C_i$ and we update $C_j$ to be the blue neighborhood of $h(v_i)$ in $C_j$.
Thus the size of $C_j$ decreases at most by a multiplicative factor of $t$ during each update.
We update the set $C_j$ so that $|C_j|$ is a multiple of $t$.
Note that each set $C_j$ is updated at most $d$ times, as we update each $C_j$ for every neighbor $v_i$ of $v_j$ in $H$ with $i<j$ and there are at most $d$ such neighbors of $v_j$, as $H$ is $d$-degenerate.
Since $|I_i| = n't^{d+1}$, we indeed get $|C_i| \geq n't$ after all updates.

If we manage to find $h(v_{n'})$, then the vertices $h(v_1),\dots,h(v_{n'})$ induce a graph that contains a blue copy of $H$ in $\chi$, as $h(v_i)$ is connected to every $h(v_j)$ with a blue edge for every $\{v_i,v_j\}\in E(H)$.
Note that $h(v_i) \in C_i \subseteq I_i$ for every $i \in [n']$.

Thus it suffices to consider the case when we cannot find the vertex $h(v_i)$ in some step~$i$.
That is, there is a neighbor $v_j$ of $v_i$ in $H$ with $i<j$ such that $C_i$ contains a set $W$ of $t$ vertices, each having at most 
$|C_j|/t-1$ blue neighbors in $C_j$.
Then, for each $w \in W$, we remove all blue neighbors of $w$ from $C_j$.
The total number of vertices that stay in $C_j$ after the removal is at least $|C_j| - t \cdot (|C_j|/t - 1) = t$.
Together with $W$, this $t$-tuple of vertices induces a red copy of $K_{t,t}$ in $\chi$ between $C_i \subseteq I_i$ and $C_j \subseteq I_j$.
\end{proof}

We now proceed with the proof of Theorem~\ref{thm-superexponentialBound}.
The proof is based on a probabilistic argument, which uses the following Chernoff-type inequality.

\begin{theorem}[Chernoff bound~\cite{mitzUpf17}] 
\label{thm-chernoff}
Let $X = \sum_{i=1}^k X_i$ be a random variable, where $\Pr[X_i = 1]=p_i$ and $\Pr[X_i = 0] = 1-p_i$ for every $i \in [k]$ and all $X_i$ are independent.
Let $\mu = \mathbb{E}(X) = \sum_{i=1}^k p_i$.
Then, for every $\delta \in (0,1)$,
\[\Pr[X \leq (1-\delta)\mu] \leq e^{-\mu\delta^2/2}.\]
\end{theorem}

We now state and prove the last auxiliary result needed in the proof of Theorem~\ref{thm-superexponentialBound}.

\begin{lemma}
\label{lem-orderingSuperexponentialBound}
Let $\mathfrak{G}$ be a bipartite edge-ordered graph with $m$ edges and with both parts having $n$ vertices.
For positive integers $t$ and $M$ that satisfy
\[\binom{M}{t}^2 \cdot e^{-t^2/(3n^2m!)} < 1,\]
there is an edge-ordering $<$ of $K_{M,M}$ such that every copy of $(K_{t,t},<)$ in $(K_{M,M},<)$ contains a copy of $\mathfrak{G}$ as an edge-ordered subgraph.
\end{lemma}
\begin{proof}
Let $<$ be the ordering of the edges of $K_{M,M}$ chosen independently and uniformly from the set of all edge-orderings of $K_{M,M}$.
The probability that a copy of $(K_{n,n},<)$ in $(K_{M,M},<)$ contains a copy of $\mathfrak{G}$ is at least $1/m!$.
For a copy of $K_{t,t}$ in $K_{M,M}$, fix a decomposition of $K_{t,t}$ into copies $B_1,\dots,B_k$ of $K_{n,n}$ where any two of them share at most a single edge.
Note that $k \geq (t/n)^2$, as we can partition each part of $K_{t,t}$ into $t/n$ sets of size $n$ and then consider copies of $K_{n,n}$ induced by these parts.
For $i \in [k]$, let $X_i$ be the random variable such that $X_i = 1$ if $(B_i,<)$ contains a copy of $\mathfrak{G}$ and let $X_i=0$ otherwise.
Then $\Pr[X_i=1]  \geq 1/m!$.
Since any two copies $B_i$ and $B_j$ share at most a single edge, all the variables $X_i$ are independent.
Let $X=\sum_{i=1}^kX_i$ and note that 
\[\mathbb{E}(X) = \sum_{i=1}^k \Pr[X_i = 1] \geq k/m! \geq \frac{t^2}{n^2m!}.\]
Clearly, the probability that a copy of $(K_{t,t},<)$ in $(K_{M,M},<)$ does not contain a copy of $\mathfrak{G}$ is at most $\Pr[X = 0]$.
By Theorem~\ref{thm-chernoff}, 
\[\Pr[X=0] \leq e^{-\mathbb{E}(X)/3}\leq e^{-t^2/(3n^2m!)}.\]
The number of copies of $K_{t,t}$ in $K_{M,M}$ is $\binom{M}{t}^2$.
The expected number of copies of $(K_{t,t},<)$ in $(K_{M,M},<)$ that do not contain a copy of $\mathfrak{G}$ is thus at most 
\[\binom{M}{t}^2 \cdot e^{-t^2/(3n^2m!)},\]
which is less than $1$ according to our assumptions.
Thus there is an edge-ordering $<$ of $K_{M,M}$ such that every copy of $(K_{t,t},<)$ in $(K_{M,M},<)$ contains a copy of $\mathfrak{G}$.
\end{proof}

We now combine Lemmas~\ref{lem-greedyEmbedding} and~\ref{lem-orderingSuperexponentialBound} to prove Theorem~\ref{thm-superexponentialBound}.

\begin{proof}[Proof of Theorem~\ref{thm-superexponentialBound}]
Let $\mathfrak{H} = (H,\prec_1)$ be a $d$-degenerate edge-ordered graph on $n'$ vertices and let $\mathfrak{G}=(G,\prec_2)$ be a bipartite edge-ordered graph with $m$ edges and with both parts having $n$ vertices.
We set $t=3n^{10}m!$, $N = (n')^2t^{d+1}$, and $M = N/n'$.
Assuming $d \leq n$ and $n' \leq t^{d+1}$, we show
\[\overline{R}_e(\mathfrak{H},\mathfrak{G}) \leq N.\]

We construct an edge-ordered complete graph $\mathfrak{K}_N=(K_N, <)$ such that every red-blue coloring of the edges of $\mathfrak{K}_N$ contains either a blue copy of $\mathfrak{H}$ or a red copy of $\mathfrak{G}$.
Let $v_1, \dots, v_{n'}$ be a vertex-ordering of $H$ such that each $v_j$ has at most $d$ neighbors $v_i$ with $i<j$.
Such an ordering exists, as $H$ is $d$-degenerate.
For every vertex $v_i$ of $H$, let $I_i$ be a set of vertices such that $|I_i| = M = n't^{d+1}$ and let the disjoint union $I_1 \cup \cdots \cup I_{n'}$ be the vertex set of $\mathfrak{K}_N$.

By Lemma~\ref{lem-orderingSuperexponentialBound}, if $\binom{M}{t}^2 \cdot e^{-t^2/(3n^2m!)} < 1$, then there is an edge-ordering $<'$ of $K_{M,M}$ such that every copy of $(K_{t,t},<')$ in $(K_{M,M},<')$ contains a copy of $\mathfrak{G}$.
We have 
\[\binom{M}{t}^2  \leq M^{2t} = (n')^{2t}t^{2t(d+1)},\]
thus it suffices to show that the expression
\[(n')^{2t}t^{2t(d+1)} \cdot e^{-t^2/(3n^2m!)} \leq  e^{2t\log{n'} + 2t(d+1)\log{t} - t^2/(3n^2m!)}\]
is less than $1$. 
From the choice of $t$ and the fact $m \leq n^2$, we have $\log{t} \leq 5n^2\log{n}$ for $n \geq 2$.
Also, since $\log{n'} \leq (d+1)\log{t}$ and $d \leq n$, we can bound the exponent in the above expression from above by 
\[4t(d+1)\log{t} - t^2/(3n^2m!) \leq t(20(d+1)n^2\log{n} - n^8) < 0\]
for $n \geq 2$.
Thus there is an edge-ordering $<'$ of $K_{M,M}$ such that every copy of $(K_{t,t},<')$ in $(K_{M,M},<')$ contains a copy of $\mathfrak{G}$.

We now define the edge-ordering $<$ of $K_N$.
Let $\prec'_1$ be an arbitrary linear ordering of the edges of $K_{n'}$ with the same vertex set as $H$ such that $\prec'
_1$ contains $\prec_1$.
For two edges $e=\{u,v\}$ and $f=\{x,y\}$ with $u \in I_i$, $v \in I_j$ and $x \in I_k$, $y \in I_l$, where $i \neq j$, $k \neq l$ and $|\{i,j\} \cap \{k,l\}| \leq 1$, we set $e < f$ if and only if $\{v_i,v_j\} \prec'_1 \{v_k,v_l\}$.
That is, $<$ is a \emph{blow-up} of $\prec'_1$ on $I_1,\dots,I_{n'}$.
For all $i$ and $j$ with $1 \leq i < j \leq n'$, we let $<$ be the ordering $<'$ on the complete bipartite graph induced by $I_i \cup I_j$.
Finally, we order the rest of the edges of $K_N$ arbitrarily, obtaining the edge-ordered graph $\mathfrak{K}_N = (K_N,<)$.

Let $\chi$ be a red-blue coloring of the edges of $\mathfrak{K}_N$.
By Lemma~\ref{lem-greedyEmbedding}, there is a blue copy of $H$ in $\chi$ with an image of each $v_i$ in $I_i$ or a red copy of $K_{t,t}$ in $\chi$ with each part contained in a different set $I_i$.
In the first case, since copy of each $v_i$ lies in $I_i$ and $<$ is a blow-up of $\prec'_1$ on $I_1,\dots,I_{n'}$, this copy of $H$ has edge-ordering isomorphic to $\prec_1$ and we obtain a blue copy of~$\mathfrak{H}$ in~$\chi$.

In the second case, we have a red copy of $K_{t,t}$ between $I_i$ and $I_j$.
Since $<$ corresponds to the ordering $<'$ on the edges between $I_i$ and $I_j$ and, by the choice of $<'$, all copies of $(K_{t,t},<')$ between $I_i$ and $I_j$ contain a copy of $\mathfrak{G}$, we obtain a red copy of $\mathfrak{G}$ in~$\chi$.
\end{proof}

The proof of Theorem~\ref{thm-polynomialBound} is also carried out using Lemma~\ref{lem-greedyEmbedding}.
However, since we are working with lexicographically edge-ordered graph, we can order edges between two sets $I_i$ and $I_j$ lexicographically and use the fact that $\mathfrak{G}$ is an edge-ordered subgraph of $\mathfrak{K}^{lex}_{m,n}$.
The rest of the proof of Theorem~\ref{thm-polynomialBound} is then analogous to the proof of Theorem~\ref{thm-superexponentialBound}.

\begin{proof}[Proof of Theorem~\ref{thm-polynomialBound}]
Let $\mathfrak{H}$ be a $d$-degenerate edge-ordered graph on $n'$ vertices and let $\mathfrak{G}$ be an edge-ordered subgraph of $\mathfrak{K}^{lex}_{n,n}$.
We set $N = (n')^2 n^{d+1}$ and show that $\overline{R}_e(\mathfrak{G},\mathfrak{H}) \leq N$ by constructing an edge-ordered complete graph $\mathfrak{K}_N=(K_N, <)$ such that every red-blue coloring of the edges of $\mathfrak{K}_N$ contains either a blue copy of $\mathfrak{H}$ or a red copy of $\mathfrak{G}$.

Letting $v_1, \dots, v_{n'}$ be a vertex-ordering of $H$ such that each $v_j$ has at most $d$ neighbors $v_i$ with $i<j$, for every vertex $v_i$ of $H$, we let $I_i$ be a set of vertices such that $|I_i| = N/n'$.
We again let the disjoint union $I_1 \cup \cdots \cup I_{n'}$ be the vertex set of $\mathfrak{K}_N$.
Let $\prec'_1$ be an arbitrary linear ordering of the edges of $K_{n'}$ with the same vertex set such that $\prec'
_1$ contains $\prec_1$.
We let $<$ be the edge-ordering of $K_N$ that is a blow-up of $\prec'_1$ on $I_1,\dots,I_{n'}$ and we order edges between two sets $I_i$ and $I_j$ so that they determine a copy of $\mathfrak{K}^{lex}_{|I_i|,|I_j|}$.

Again, for every red-blue coloring $\chi$ of the edges of $\mathfrak{K}_N$, Lemma~\ref{lem-greedyEmbedding} gives a blue copy of~$H$ in $\chi$ with an image of each $v_i$ in $I_i$ or a red copy of $K_{n,n}$ in $\chi$ with each part contained in a different set $I_i$.
In the first case, we obtain a blue copy of~$\mathfrak{H}$ as before.
In the second case, since the edges between any two sets $J_i \subseteq I_i$ and $J_j \subseteq I_j$ determine a copy of $\mathfrak{K}^{lex}_{|J_i|,|J_j|}$, we obtain a red copy of $\mathfrak{G}$, as $\mathfrak{G}$ is an edge-ordered subgraph of $\mathfrak{K}^{lex}_{n,n}$.
\end{proof}

\section{Lexicographic edge-ordered Ramsey numbers}
\label{sec-lexicographic}

Here, we include proofs of all statements about lexicographically edge-ordered graphs from Section~\ref{sec-ourResults}.
We start with a simple proof of Lemma~\ref{lem-lexicographicVertex}.

\begin{proof}[Proof of Lemma~\ref{lem-lexicographicVertex}]
For a lexicographically edge-ordered graph $\mathfrak{G}=(G,\prec)$ with the vertex set $V$, let $f \colon V \to \{1,\dots,|V|\}$ be any one-to-one correspondence consistent with $\mathfrak{G}$ and let $\mathcal{G}_f$ be the vertex-ordering of $G$ determined by~$f$.
More specifically, the vertex-ordering $\mathcal{G}_f=(G,<')$ is chosen such that $u<'v$ if and only if $f(u)<f(v)$.
Without loss of generality, the edge-ordered graph $\mathfrak{K}^{lex}_N$ has the vertex set $[N]$.

We show that every copy of $\mathcal{G}_f$ in $\mathcal{K}_N$ on $[N]$ determines a copy of $\mathfrak{G}$ in $\mathfrak{K}^{lex}_N$.
Let $i \colon V \to [N]$ be an inclusion witnessing that $\mathcal{G}_f$ is an ordered subgraph of $\mathcal{K}_N$.
That is, $i(u)<i(v)$ if and only if $f(u)<f(v)$ for all $u,v \in V$.
Then $\mathfrak{G}$ is an edge-ordered subgraph of $\mathfrak{K}^{lex}_N$, since, for edges $\{u,v\}$ and $\{w,t\}$ of $G$ with $f(u)<f(v)$ and $f(w)<f(t)$, we have $\{u,v\} \prec \{w,t\}$ if and only if $f(u)<f(w)$ or $(f(u)=f(w) \;\&\; f(v) < f(t))$.
This is true if and only if $i(u)<i(w)$ or $(i(u)=i(w) \;\&\; i(v) < i(t))$, which corresponds to $\{i(u),i(v)\}$ preceding $\{i(w),i(t)\}$ in $\mathfrak{K}^{lex}_N$.

Thus, given a $2$-coloring $\chi$ of the edges of $\mathfrak{K}^{lex}_N$ with $N = \overline{R}(\mathcal{G}_f)$, if we consider $\chi$ as a $2$-coloring of the edges of $\mathcal{K}_N$, we obtain a monochromatic copy of $\mathcal{G}_f$ in $\chi$.
Since every copy of $\mathcal{G}_f$ in $\mathcal{K}_N$ on $[N]$ determines a copy of $\mathfrak{G}$ in $\mathfrak{K}^{lex}_N$, we obtain a monochromatic copy of $\mathfrak{G}$ in $\chi$.
Since $f$ is an arbitrary mapping consistent with $\mathfrak{G}$, it follows that $\overline{R}_{lex}(\mathfrak{G}) \leq \overline{R}(\mathcal{G}_f)$ and finishes the proof of Lemma~\ref{lem-lexicographicVertex}.
\end{proof}

Using Lemma~\ref{lem-lexicographicVertex}, we now present a proof of Proposition~\ref{prop-lexicographicPath}.
That is, we show that $\overline{R}_{lex}(\mathfrak{P}_n) \leq 2n-3+\sqrt{2n^2-8n+11}$ for every $n > 2$, where $\mathfrak{P}_n$ is the edge-monotone path on $n$ vertices.

\begin{proof}[Proof of Proposition~\ref{prop-lexicographicPath}]
The proof is based on the same idea used in~\cite[Proposition~15]{bckk13}.
Let $N$ be a positive integer and assume that there is a $2$-coloring of the edges of $\mathfrak{K}^{lex}_N$ with no monochromatic copy of $\mathfrak{P}_n$.
Let $[N]$ be the vertex set of $\mathfrak{K}^{lex}_N$ and assume that the vertex-order $<$ is determined by the lexicographic edge-ordering of $K_N$.
Without loss of generality, we assume that at least half of the edges with one vertex from $\left[\left\lceil \frac{N}{2}\right\rceil\right]$ and the other one in $\left\{\left\lceil \frac{N}{2}\right\rceil+1,\dots,N\right\}$ are colored red.
Let $M$ be an $\lceil \frac{n}{2} \rceil \times \lfloor \frac{n}{2} \rfloor$ matrix with entries from $\{0,1\}$ such that it contains $1$-entries on positions $(i,i)$ and $(i+1,i)$ and $0$-entries otherwise.
Note that $M$ naturally corresponds to an ordered path $\mathcal{P} = (P,<)$ on vertices $1<\cdots<n$ by connecting $i$ and $\lceil n/2 \rceil + j$ with an edge if and only if the $(i,j)$ entry of $M$ is $1$.
Also the identity on the vertex set $[n]$ of $\mathcal{P}$ is consistent with the lexicographic ordering of the edges of $\mathcal{P}$; see Figure~\ref{fig-monotone} for an example.
Thus the matrix $M$ represents the edge-monotone path $\mathfrak{P}_n$.

Let $A$ be an $\left\lceil \frac{N}{2} \right\rceil \times \left\lfloor \frac{N}{2} \right\rfloor$ matrix with entries from $\{0,1\}$ such that $A$ contains $1$-entries on positions $\left(i,j\right)$, where $\{i,\left\lceil \frac{N}{2} \right\rceil+j\}$ is a red edge of $\mathfrak{K}^{lex}_N$, and $0$-entries otherwise.
We say that $A$ \emph{contains} the matrix $M$ if $A$ contains a submatrix that has $1$-entries at all the positions where $M$ does.
Observe that, since $M$ represents the edge-monotone path $\mathfrak{P}_n$, the matrix $A$ cannot contain $M$ as otherwise $\mathfrak{K}^{lex}_N$ contains a red copy of $\mathfrak{P}_n$.

It follows from a result of F\"{u}redi and Hajnal~\cite{fur92} (see also Lemma~16 in~\cite{bckk13}) that $A$ contains at most 
\begin{align*}
\left(\left\lfloor \frac{n}{2}\right\rfloor-1\right)\left\lceil \frac{N}{2}\right\rceil+\left(\left\lceil \frac{n}{2}\right\rceil-1\right)\left\lfloor \frac{N}{2}\right\rfloor
&-\left(\left\lceil \frac{n}{2}\right\rceil-1\right)\left(\left\lfloor \frac{n}{2}\right\rfloor-1\right) \\
&\le \frac{2nN+4n-4N-3-n^2}{4}
\end{align*}
$1$-entries, as $M$ is so-called \emph{minimalist} matrix.
On the other hand, since red was the major color among edges between $\left[\left\lceil \frac{N}{2}\right\rceil\right]$ and $\left\{\left\lceil \frac{N}{2}\right\rceil+1,\dots,N\right\}$, we have at least $\frac{1}{2}\left\lceil \frac{N}{2}\right\rceil\left\lfloor \frac{N}{2}\right\rfloor \geq (N^2-1)/8$ such red edges and thus also at least that many $1$-entries in $A$.
Altogether, we have the inequality $(2nN+4n-4N-3-n^2)/4 < (N^2-1)/8$, which gives $N \leq 2n-4+\sqrt{2n^2-8n+11}$.
\end{proof}

\section{Open problems}
\label{sec-openProblems}

Many questions about edge-ordered Ramsey numbers remain open, for example proving a better upper bound on edge-ordered Ramsey numbers than the one obtained in the proof of Theorem~\ref{thm-edgeOrderedRamseyFinite}.
Since edge-ordered Ramsey numbers do not increase by removing edges of a given graph, it suffices to focus on edge-ordered complete graphs.
It is possible that edge-ordered Ramsey numbers of edge-ordered complete graphs do not grow significantly faster than the standard Ramsey numbers.

\begin{problem}
\label{prob-lowerBoundClique}
Is there a constant $C$ such that, for every $n \in \mathbb{N}$ and every edge-ordered complete graph $\mathfrak{K}_n$ on $n$ vertices, we have $\overline{R}_e(\mathfrak{K}_n) \leq 2^{C n}$?
\end{problem}

We note that, very recently, Fox and Li~\cite{foxLi19} showed that $\overline{R}_e(\mathfrak{H}) \leq 2^{100n^2\log^2{n}}$ for every edge-ordered graph $\mathfrak{H}$ on $n$ vertices.

It might also be interesting to consider sparser graphs, for example graphs with maximum degree bounded by a fixed constant, and try to prove better upper bounds on their edge-ordered Ramsey numbers.

We have no non-trivial results about lower bounds on edge-ordered Ramsey numbers and even lexicographically edge-ordered Ramsey numbers.
Proving lower bounds for the latter might be simpler, as one has to consider only the lexicographic edge-ordering of the large complete graph.
Is there a class of graphs with maximum degree bounded by a fixed constant such that their corresponding edge-ordered Ramsey numbers grow superlinearly in the number of vertices?
If so, then such a result would contrast with the famous result of Chv\'{a}tal, R\"{o}dl, Szemer\'{e}di, and Trotter~\cite{crst83}, which says that Ramsey numbers of bounded-degree graphs grow at most linearly in the number of vertices.

Another interesting open problem is to determine the growth rate of the function $T(n)$ from Theorem~\ref{thm-canonical}.
The current upper bound on $T(n)$ is quite large as the proof of Ne\v{s}et\v{r}il and R\"{o}dl~\cite{nesetrilRodl17} uses Ramsey's theorem for quadruples and $6!=720$ colors.

Finally, we showed that the inequality in Lemma~\ref{lem-lexicographicVertex} is not always tight using examples with stars, where both sides of the inequality differ by $1$.
It is a natural question to ask how wide this gap can be.
In particular, is there a class of graphs for which the ratio between both sides of the inequality in Lemma~\ref{lem-lexicographicVertex} is arbitrarily large?

\paragraph{Acknowledgement}

We would like to thank to the referee for his/her careful reading and comments, to Jan Kyn\v{c}l for interesting discussions about canonical edge-orderings and also to the organizers of the Eml\'{e}kt\'{a}bla Workshop 2018, where the research was initiated.

\bibliography{bibliography}	
\bibliographystyle{plain}

\end{document}